\documentclass{amsart}
\usepackage{amsmath, amsthm, amssymb, amscd, amsfonts,enumerate, amsthm}
\usepackage[utf8]{inputenc}
\usepackage[mathscr]{eucal}
\usepackage{color}
\usepackage{float}
\usepackage{cite}
\newtheorem{theorem}{Theorem}[section]
\newtheorem{lemma}[theorem]{Lemma}
\newtheorem{proposition}[theorem]{Proposition}
\newtheorem{corollary}[theorem]{Corollary}
\newtheorem{definition}[theorem]{Definition}

\usepackage{longtable}
\usepackage{threeparttablex}

\theoremstyle{remark}

\newtheorem{exm}{Example}

\title{Quasi-Primary Spectrum of a Commutative Ring and a Sheaf of Rings }

\makeatletter
\renewcommand*{\@fnsymbol}[1]{\ifcase#1\or1\else\@arabic{\numexpr#1\relax}\fi}

	\author{Zehra B\.{ı}lg\.{ı}n}
	\address{Philosophy Department, Istanbul Medeniyet University, Kadikoy,
		Istanbul, TURKEY}
	\email{zehrabilgin.zb@gmail.com}
	
		\author{Nesl\.{ı}han Ayşen Özk\.{ı}r\.{ı}şc\.{ı}*}
		\address{Mathematics Department, Yıldız Technical University, Esenler,
			Istanbul, TURKEY}
		\email{aozk@yildiz.edu.tr}

\begin{document}
	
		\date{}
	\maketitle
	
	\let\thefootnote\relax
	\footnote {
		
     *Corresponding author.} 
	
	\begin{abstract}
		 In this work, the set of quasi-primary ideals of a commutative ring with identity is equipped with a topology and is called quasi-primary spectrum. Some topological properties of this space are examined. Further, a sheaf of rings on the quasi-primary spectrum is constructed and it is shown that this sheaf is the direct image sheaf with respect to the inclusion map from the prime spectrum of a ring to the quasi-primary spectrum of the same ring. 
	\end{abstract}
	
	\section{Introduction}
	The set of all prime ideals of a commutative ring $R$, called \emph{prime spectrum} of $R$, denoted by $Spec(R)$, is a well-known concept in commutative algebra. This set is equipped with the famous Zariski topology, where closed sets are defined as $V(I)=\{P\in Spec(R): I\subseteq P\}$ for any ideal $I$ of $R$. Topological properties of $Spec(R)$ are widely examined throughout the years and can be found in many of the standard commutative algebra and algebraic geometry references. Besides, there is a famous sheaf construction, named \emph{the structure sheaf}, on $Spec(R)$ which is a very useful tool to connect algebraic geometry and commutative algebra. For details of the structure sheaf the reader may consult \cite{H}, \cite{S} and \cite{U}.
	
	In \cite{OKK}, the authors generalized the Zariski topology on $Spec(R)$ to the set of primary ideals of a commutative ring $R$, denoted by $Prim(R)$, and they called it \emph{primary spectrum} of $R$. They defined the closed sets as $$V_{rad}(I)=\{Q\in Prim(R): I\subseteq \sqrt{Q}\}$$ for any ideal $I$ of $R$ where $\sqrt{Q}$ denotes the radical of $Q$. They showed that these closed sets satisfy axioms of a topology on $Prim(R)$. They investigated some topological properties of this space and compared them with the well-known properties of $Spec(R)$. We note that, since any prime ideal is primary and equal to its radical, the space $Spec(R)$ is in fact a subspace of $Prim(R)$. 
	
	When \cite{OKK} is examined in detail, it can be realized that the given topological construction depends only the fact that the radical of a primary ideal is prime. So, this topology is in fact valid on a much larger set, the set of ideals whose radicals are prime. These type of ideals are first introduced by L. Fuchs in \cite{F}. He named them as \emph{quasi-primary} ideals. We aim to investigate the set of quasi-primary ideals of a commutative ring $R$ equipped with a topology similar to the one defined in \cite{OKK} and to construct a sheaf of rings on this topological space.

In Section 2, after giving some general topological properties of quasi-primary spectrum, we examine irreducibility and irreducible components of this space. Besides we investigate disconnectedness of the space and finally show that the dimension of the quasi-primary spectrum of a Noetherian local ring is finite. In Section 3, we construct a sheaf of rings on the quasi-primary spectrum of and prove that this sheaf is actually the direct image sheaf under the inclusion map from the prime spectrum to the quasi-primary spectrum. Further, we conclude that this sheaf is in fact a scheme.
	\section{The Quasi-Primary Spectrum of a Ring}
		Throughout this paper all rings are commutative with identity. In this section, we define a topology on the set of all quasi-primary ideals of a ring and examine some properties of this topological space. 
		
		First, we give some (known) properties of quasi-primary ideals we need in the rest of the paper. 
		Let $R$ be a ring. Following \cite{F}, an ideal $I$ of $R$ is called \emph{quasi-primary} if the radical of $I$, denoted by $\sqrt{I}$ is prime. In this section 
		\begin{lemma}
			Let $I$ be an ideal of a ring $R$ and $S$ a multiplicative subset of $R$. Denote the localization of $R$ with respect to $S$ by $R_S$.
			\begin{enumerate}[(i)]
				\item If  $I$ is primary, then $I$ is quasi-primary.
				\item If $I$ is a quasi-primary ideal, the $I$ has only one minimal prime ideal.
				\item If $IR_S$ is a quasi-primary ideal of $R_S$, then $IR_S\cap R$ is a quasi-primary ideal of $R$.
				\item If $I$ is a quasi-primary ideal of $R$ such that $\sqrt{I}\cap S=\varnothing$, then $IR_S$ is a quasi-primary ideal of $R_S$.
			\end{enumerate}
		\end{lemma}
		\begin{proof}
			(i) and (ii) is obvious. For (iii) and (iv) it is enough to observe that $\sqrt{I}R_S=\sqrt{IR_S}$.
		\end{proof}
	
	  Let
	$$QPrim(R)=\{I\subseteq R : I \text{ a quasi-primary ideal}\}.$$
	For any subset $S$ of $R$, set
	$$V_q(S)=\{Q\in QPrim(R): S\subseteq \sqrt{Q}\}.$$
	Observe that, for any subset $S$ of $R$, if $I=(S)$ we have $V_q(S)=V_q(I)$. If $S=\{a\}$ for $a\in R$, we write $V_q(S)=V_q({a})$.
	
	In \cite{OKK}, the authors defined a topology on the set $Prim(R)$ of primary ideals of a commutative ring using the sets $V_{rad}(I)=\{Q\in Prim(R): I\subseteq \sqrt{Q}\}$ as the closed sets. In this construction they only used the property that a primary ideal has a prime radical. So we realized that the topology axioms for closed sets are in fact satisfied by the sets
	$$V_q(I)=\{Q\in QPrim(R): I\subseteq \sqrt{Q}\}$$
	where $I$ any ideal of $R$. Thus,  $QPrim(R)$ is a topological space with closed sets $V_q(I)$ where $I$ is any ideal of $R$. Since any primary ideal is quasi-primary, we have $Prim(R)$ as a subspace of $QPrim(R)$. 
	
	For the sake of completeness, we note some properties (from Proposition \ref{E:gs} to Corollary \ref{E:ge}) of $V_q(I)$ without proofs. For details, see \cite{OKK}.
	
	\begin{proposition}\label{E:gs}
		Let  $I,J$ be ideals of $R$ and $\{I_\lambda\}_{\lambda\in\Lambda}$ a family of ideals of $R$. Then the followings hold:
		\begin{enumerate}[(i)]
			\item If  $I\subseteq J$, then $V_q(J)\subseteq V_q(I)$.
			\item $V_q(0)=QPrim(R)$ and $V_q(R)=\varnothing$.
			\item $V_q(I\cap J)=V_q(IJ)=V_q(I)\cup V_q(J)$.
			\item $V_q(\sum_{\lambda\in \Lambda}I_\lambda)=\bigcap_{\lambda\in \Lambda}V_q(I_\lambda)$.
			\item $V_q(I)=V_q(\sqrt{I})$.
		\end{enumerate}

		\end{proposition}

	\begin{corollary}
		The family $\{V_q(I): I \text{ is an ideal of } R\}$ satisfies the axioms of closed sets of a topology on $QPrim(R)$. 
	\end{corollary}
This topology is called \emph{Zariski topology} on $QPrim(R)$ and the space $QPrim(R)$ is named as the \emph{quasi-primary spectrum} of $R$.
We note that any open set in $QPrim(R)$ is of the form $QPrim(R)\backslash V_q(S)$ for some subset $S$ of $R$. 

Set $U_a=QPrim(R)\backslash V_q(a)$ for any $a\in R$. 
\begin{theorem}
	Let $R$ be a ring. The family $\{U_a\}_{a\in R}$ is a base for the Zariski topology on $QPrim(R)$.
\end {theorem} 
Note that $U_0=\varnothing$ and $U_r=QPrim(R)$ for every unit $r\in R$.
\begin{theorem}
	Let $R$ be a ring and $a,b\in R$. The followings hold:
	\begin{enumerate}[(i)]
		\item $\sqrt{(a)}=\sqrt{(b)}$ if and only if $U_a=U_b$.
		\item $U_{ab}=U_a\cap U_b$.
		\item $U_a=\varnothing$ if and only if $a$ is nilpotent.
		\item $U_a$ is quasi-compact.
	\end{enumerate}

\end{theorem}
\begin{corollary}\label{E:ge}
	Let $R$ be a ring. $QPrim(R)$ is quasi-compact.
\end{corollary}

Quasi-primary ideals firstly introduced and examined thoroughly in \cite{F}. It is generally studied on rings satisfying maximal condition; in other words, that every ascending chain of ideals is finite. It is also noted that the quasi-primary ideals in rings satisfying maximal condition can be characterized as follows: A quasi-primary ideal is either a power of a prime ideal or an intermediate ideal between two powers of one and the same prime ideal. In the view of this fact, the following theorem is given for rings satisfying maximal condition.
 
\begin{theorem}	\cite[Theorem 4]{F} \label{E:pq}
	If $ Q_1 $ and $ Q_2 $ are quasi-primary ideals having the radicals $ P_1 $ and $ P_2 $ respectively, and $ P_1\subseteq P_2 $. Then $ Q_1Q_2 $ is also quasi-primary having the radical $ P_1. $
\end{theorem}

\begin{theorem} \label{E:vq}
	Let $ R $ be a ring satisfying maximal condition and $ Q_1 $ and $ Q_2 $ be quasi-primary ideals of $ R $ such that $ Q_1\subseteq Q_2 $. If $ Q_1\in V_q(I) $, then $ Q_1Q_2\in V_q(I) $ for any ideal $ I $ of $ R $. 
\end{theorem}
\begin{proof}
	It follows from Theorem \ref{E:pq}.
\end{proof}

It is known that Theorem \ref{E:pq} has no analogue in primary ideal theory. Similarly, Theorem \ref{E:vq} does not valid for the primary spectrum as can be seen in the following example.

\begin{exm}
	Consider the residue class ring $ R=K[X_1,X_2,X_3]/(X_1X_3-X_2^2) $ where $ K $ is a field. It is clear that $ R $ satisfies maximal condition. Let $ x_i $ denote the natural image of $ X_i $ in $ R $ for each $ i=1,2,3 $. Then, the ideal $ P=(x_1,x_2) $ is a prime ideal of $ R $ but $ P^2 $ is not primary \cite[Example 4.12]{SH}. It is trivial that $ P^2 $ is a quasi-primary ideal of $ R $. Now take $ Q_1=Q_2=P. $ Then, we can see that $ P\in V_q(P)\cap V_{rad}(P) $, however $ P^2\in V_q(P)\setminus V_{rad}(P). $
\end{exm}

Now, let us determine the closure of a point $Q\in QPrim(R)$. The closure $Cl(Q)$ of $Q$ is
$$Cl(Q)=\bigcap_{Q\in V_q(S)}V_q(S)=\bigcap_{S\subseteq \sqrt{Q}}V_q(S)=V_q(Q).$$

\begin{definition}
	A topological space $X$ is irreducible if $X$ is nonempty and $X$ cannot be written as a union of two proper closed subsets, or equivalently, any two nonempty open subsets of $X$ intersect.
\end{definition}

\begin{theorem}
	Let $ R $ be a ring. Then $ QPrim(R) $ is an irreducible space if and only if the nilradical of $ R $, $ \mathcal{N}(R) $ is quasi-primary.
\end{theorem}
\begin{proof}
	Let $ \mathcal{N}(R) $ be a quasi-primary ideal and $ U,V $ be two non-empty open subsets of $ QPrim(R)$. Suppose that $ Q_1\in U\backslash  V$. Then there exists a subset $ S $ of $ R $ such that $ U=X\setminus V_q(S) $. This implies $ Q_1\notin V_q(S) $, that is, $ S\nsubseteq \sqrt{Q_1} $. Then we get $ S\nsubseteq \sqrt{\mathcal{N}(R)} $ since $ \mathcal{N}(R)  \subseteq \sqrt{Q_1} $. Hence $ \mathcal{N}(R) \notin V_q(S) $. Then, we obtain $ \mathcal{N}(R) \in U $. In a similar way, we get $ \mathcal{N}(R)\in V $. Consequently, $ U\cap V\neq  \emptyset $.
	For the other part, let $ QPrim(R) $ be an irreducible space. Assume that $ \mathcal{N}(R) $ is not quasi-primary. Then $ \sqrt{\mathcal{N}(R)} $ is not prime. Then there exist $ a,b\in R $ such that $ a,b\notin \sqrt{\mathcal{N}(R)} $ but $ ab\in \mathcal{N}(R) $. Since $ a\in \sqrt{\mathcal{N}(R)}$, $V_q(a)\neq QPrim(R) $, that is, $ U_a\neq\emptyset $. Similarly, we get $ U_b\neq \emptyset $. Since $ ab\in \mathcal{N}(R) $ we have $ U_{ab}=\emptyset $. It follows that $ U_a\cap U_b=U_{ab}=\emptyset $ for  two non-empty open subsets $ U_a $ and $ U_b $. Hence $ QPrim(R) $ is not irreducible.
\end{proof}
There is a one to one correspondence between points of $ QPrim(R) $ and irreducible closed subsets of $ QPrim(R) $. The next theorem gives that correspondence.

\begin{theorem}\label{E:correspondence}
	Let $ Y $ be a subset of $ QPrim(R) $. Then $ Y $ is an irreducible closed subset of $ QPrim(R) $ if and only if $ Y=V_q(Q) $ for some $ Q\in QPrim(R) $.
\end{theorem}
\begin{proof}
	Let $ Y=V_q(Q) $ for any $ Q\in QPrim(R) $. Since $ V_q(Q)=Cl({Q}) $ and $ Cl({Q}) $ is irreducible, $ Y $ is an irreducible closed subset of $ QPrim(R) $. Conversely, let $ Y $ be an irreducible closed subset of $ QPrim(R) $. Then $ Y=V_q(I) $ for some ideal $ I $ of $ R $. Now suppose that $ I\notin QPrim(R) $. Then $ \sqrt{I} $ is not prime. Then there are elements $ a,b\in R $ such that $ ab\in I $ but $ a,b\notin \sqrt{I} $. Thus, $ Y=V_q(I)\subseteq V_q(ab)=V_q(a)\cup V_q(b)$. Also $ V_q(a)\neq V_q(I) $ and $ V_q(b)\neq V_q(I) $ because of $ a,b\notin \sqrt{I} $. So we obtain $ Y $ reducible which contradicts our assumption. 
\end{proof}

Let $ I $ be an ideal of a ring satisfying maximal condition. Then, by \cite[Theorem 5]{F}, the ideal $ I $ is the intersection of a finite number of quasi-primary ideals, say $ Q_1, \dots,Q_n $ with radicals $ P_1,\dots, P_n $, respectively. Hence, $ \sqrt{I}=\bigcap^n_{i=1}\sqrt{Q_i}=\bigcap^n_{i=1}P_i $, that is, there is no prime ideal containing $ I $ other than $ P_i $'s where $ i=1, \dots, n  $. Then, for any ideal $ I $ in a ring satisfying maximal condition, every closed subset $ V_q(I) $ can be written as the finite union of irreducible closed sets, that is,  $ V_q(I)=V_{q}(P_1)\cup \dots \cup V_q(P_n) $ by Proposition \ref{E:gs} (iii) and (v). 
	\\
	
Let $V$ be a closed subset of a topological space $X$. A dense point of $V$ is called a \emph{generic point}. By the above theorem, we conclude that every irreducible closed subset of $QPrim(R)$ has a generic point.

The maximal irreducible subsets of a topological space $X$ are called \emph{irreducible components}.
\begin{theorem}
	Irreducible components of $QPrim(R)$ are the closed sets $V_q(Q)$ where $\sqrt{Q}$ is a minimal prime ideal of $R$.
\end{theorem}
\begin{proof}
	By Theorem \ref{E:correspondence}, any irreducible closed subset of $QPrim(R)$ can be written of the form $V_q(Q)$ for some quasi-primary ideal $Q$ of $R$. Assume that $V_q(Q)$ is not maximal. Then $V_q(Q)\subset V_q(Q')$ for some quasi-primary ideal $Q'$ of $R$. Since $  V_q(Q')=V_q(\sqrt{Q'})$ we have $\sqrt{Q'}\subset \sqrt{Q}$. Hence $\sqrt{Q}$ is not minimal. Conversely, assume that $\sqrt{Q}$ is not a minimal prime ideal. Then there is a prime ideal $P$ of $R$ such that $P\subset \sqrt{Q}$. Then we get $V_q(Q)\subset V_q(P)$. So $V_q(Q)$ is not a maximal irreducible set.
\end{proof}
The following lemma is easy to prove, so we left it as an exercise.
\begin{lemma}
	Let $R=R_1\times R_2\times \cdots \times R_n$ where $R_i$ are rings. Then
	$$QPrim(R)\cong QPrim (R_1)\times QPrim(R_2)\times \cdots\times QPrim(R_n).$$
\end{lemma}
\begin{theorem}
	Let $R$ be a ring. The following are equivalent:
	\begin{enumerate}[(i)]
		\item $QPrim(R)$ is disconnected.
		\item $R\cong R_2\times R_2$ where $R_1$ and $R_2$ are nonzero rings.
		\item $R$ contains an idempotent.
		
	\end{enumerate}
\end{theorem}
\begin{proof}
	(i)$\implies$(ii) Assume that $QPrim(R)$ is disconnected. Then $QPrim(R)=V_q(I)\cap V_q(J)$ for some ideals $I$ and $J$ of $R$ where $V_q(I)\cap V_q(J)
=\varnothing$. Then we have $I+J=R$ and $I\cap J=IJ$. So, we get $R=R/I\times R/J$.

(ii)$\implies$(iii) Assume that $R\cong R_2\times R_2$ where $R_1$ and $R_2$ are nonzero rings via an isomorphism $ \phi$. Then $\phi^{-1}(1,0)$ is a nontrivial idempotent of $R$.

(iii)$\implies$(i) Assume that $e\in R$ is an idempotent. Then $QPrim(R)=X_1\cup X_2$ where 
$X_1=\{Q\in QPrim(R): e\in \sqrt{Q}\}$ and $X_1=\{Q\in QPrim(R): 1-e\in \sqrt{Q}\}$. Observe that $X_1\cap X_2=\varnothing$. Thus, $QPrim(R)$ is disconnected.
	\end{proof}
	
The \emph{dimension} of a topological space $X$ is the number $n$ such that $X$ has a chain of irreducible closed sets
$$V_1\subset V_2 \subset \cdots \subset V_n$$
and no such chain more that $n$ terms.
\begin{theorem}
	Let $R$ be a Noetherian local ring. Then the dimension of $QPrim(R)$ is finite.
\end{theorem}
\begin{proof}
	Let $$X_1\subset X_2\subset \cdots \subset X_n\subset \cdots$$ be a chain of irreducible subsets. This chain can be written as
	$$V_q(Q_1)\subset V_q(Q_2)\subset \cdots \subset V_q(Q_n)\subset \cdots$$
	where $Q_i\in QPrim(R)$. Let $P_i=\sqrt{Q_i}$ for each $i$. Then we have
	$$\cdots\subset P_n\subset \cdots\subset P_2\subset P_1.$$
	By \cite[Corollary 11.11]{AM}, the dimension of $R$ is finite. So the above chain of prime ideals must terminate. Therefore the dimension of $QPrim(R)$ is finite, and in fact equal to the dimension of $R$.
\end{proof}
	\section{A Sheaf of Rings on the Quasi-Primary Spectrum}
	In this section we define a sheaf of rings on the quasi-primary spectrum.
	Let $\phi:R\rightarrow R'$ be a ring homomorphism. For any $Q\in QPrim(R')$, it is easy to show that $f^{-1}(Q)\in QPrim(R)$. So $f$ induces the map 
	$$\phi^{a}:QPrim(R')\rightarrow QPrim(R)$$
	which is called the \emph{associated map} of $\phi$.
	
	For any $A\subseteq R$ we have $(\phi^a)^{-1}(V(A))=V(\phi(A))$. So the map $\phi^a$ is continuous.
	
	Let $S\subseteq R$ be a multiplicative subset of $R$. 
	Let $\phi: R\rightarrow R_S$ be the canonical homomorphism. Since $\sqrt{IR_S}=\sqrt{I}R_S$ for any ideal $I$ of $R$, the map $\phi^a$ is an inclusion. The set $U_S=\phi^a(QPrim(R_S))$ is equal to the set of quasi-primary ideals of $R$ whose radicals are disjoint from $S$. There is a one-to-one correspondence between quasi-primary ideals of $R_S$ and quasi-primary ideals of $R$ whose radicals are disjoint from $S$. So, the space $QPrim(R_S)$ is homeomorphic to the subspace $U_S$ of $QPrim(R)$. 
	
	In particular, if $S=\{f^i: i\in\mathbb{N}\}$ then $U_S=QPrim(R)\backslash V_q(f)=U_f$. So basis sets $U_f$ are homeomorphic to $QPrim(R_f)$ where $R_f$ is the localization of $R$ with respect to the multiplicative subset $S=\{f^i: i\in\mathbb{N}\}$.
	
\begin{lemma}\label{E:cont}
     $U_a\subseteq U_b$ if and only if $a\in \sqrt{(b)}$ for any $a,b\in R$.
\end{lemma}
\begin{proof}
	Assume that $U_a\subseteq U_b$ for some $a,b\in R$. Then, for any $Q\in QPrim(R)$, we have $a\not\in\sqrt{Q}$ implies $b\not\in\sqrt{Q}$. That means $b\in\sqrt{Q}$ implies $a\in \sqrt{Q}$. Since $QPrim(R)$ contains prime ideals, this observation yields that $a\in\sqrt{(b)}$. Conversely, assume that $a\in \sqrt{(b)}$ for some $a,b\in R$. Let $q\in U_a$. Then $a\not\in \sqrt{Q}$. Since $a$ is contained in the intersection of all prime ideals that contain $b$, we obtain that $b\not\in \sqrt{Q}$. Therefore, we have $Q\in U_b$.
\end{proof}

Our aim is to construct a sheaf of rings on $QPrim(R)$. We assign to each open set $U_a$ the ring $\mathcal{F}(U_a):=R_a$, ring of quotients with respect to the multiplicative subset $\{1,a,a^2,...\}$, and define the restriction maps $$res_{U_b,U_a}:R_b\rightarrow R_a,~~~~~~r/b^m\mapsto t^mr/a^{nm}.$$
Since, by Lemma \ref{E:cont}, we have  $U_a\subseteq U_b$ if and only if $a^n=tb$ for some positive integer $n$ and $t\in R$, the map $res_{U_b,U_a}$ is well-defined.

 For an arbitrary open set $U$ of $QPrim(R)$ let
	$$\mathcal{F}(U)=\varprojlim \mathcal{F}(U_a)$$
	where the projective limit is taken over all $U_a\subseteq U$ relative to the system of homomorphisms $res_{U_b,U_a}$ for $U_a\subseteq U_b$.
	 
	For $U\subseteq V$, each family $\{v_i\}\in \mathcal{F}(V)$ consisting of $v_i\in R_{a_i}$ with $ U_{a_i}\subseteq V$ defines a subfamily $\{a_j\}$ consisting of the $a_j$ for those indexes $j$ with $ U_{a_j}\subseteq U$. Then $\{v_j\}\in \mathcal{F}(U)$. Define
	$$res_{V,U}:\mathcal{F}(V)\rightarrow \mathcal{F}(U),~~~~\{v_i\}\mapsto \{v_j\}~~.$$
	
	With this construction, $\mathcal{F}$ turns to be a sheaf of rings on $QPrim(R)$.
	In fact, this sheaf is the direct image sheaf under the inclusion map from $ Spec(R)$ into $QPrim(R)$:
	\begin{theorem} The sheaf $\mathcal{F}$ on $QPrim(R)$ is equal to the direct image sheaf $\iota_*$ under the inclusion map $\iota: Spec(R)\rightarrow QPrim(R)$. 
		\end{theorem}
	\begin{proof}
	The inclusion map $\iota$ is continuous. For any open set $U$ of $QPrim(R)$, direct image sheaf $\iota_*$ is defined as follows:
	$$\iota_*(U)=\mathcal{O}(\iota^{-1}(U))$$
	where $\mathcal{O}$ denotes the structure sheaf on $Spec(R)$. For $a\in R$, we have
	\begin{align}
	\iota^{-1}(U_a)&=\{P\in Spec{R}: \iota(P)\in U_a\}\notag\\
	&=\{P\in Spec{R}: P\in U_a\}\notag\\
	&=\{P\in Spec{R}: a\not\in \sqrt {P}=P\}\notag
	\end{align}
	The final set is a principal open set for $Spec(R)$ and the corresponding ring for this set is $R_a$. So we get
	$\iota_*(U_a)=R_a=\mathcal{F}(U_a)$.
	
	For $U_a\subseteq U_b$, we have $res_{U_b,U_a}=\rho^{X_b}_{X_a}$ where $\rho^{X_b}_{X_a}$ is the restriction map from principal open set $X_b$ to $X_a$ of $Spec(R)$ with respect to the structure sheaf $\mathcal{O}$.
	Thus the sheafs $\mathcal{F}$ and $\iota_*\mathcal{F}$ are the same.
	\end{proof}
Similar to the structure sheaf $\mathcal{O}$ on $Spec(R)$, the stalk $\mathcal{F}_Q$ of $\mathcal{F}$ at a point $Q\in QPrim(R)$ is $R_{\sqrt{Q}}$. Therefore, we conclude that $(QPrim(R),\mathcal{F})$ is a locally ringed space. Finally, since $\mathcal{F}$ is the direct image sheaf of $\mathcal{O}$ under the inclusion map $\iota: Spec(R)\rightarrow QPrim(R)$, it is easily seen that $(QPrim(R),\mathcal{F})$ is a scheme.

	\end{document}